\documentclass[final,nomarks]{dmtcs-episciences}

\usepackage[utf8]{inputenc}
\usepackage{subfigure}

\usepackage{url}
\usepackage{cmll}
\usepackage{amsmath}
\usepackage{amsthm}
\usepackage{amssymb}
\usepackage{graphicx}
\usepackage{hyphenat}

\usepackage{enumitem}
\setlist{font=\normalfont}

\usepackage{tikz}
\usetikzlibrary{calc,decorations.markings,arrows.meta,positioning}

\usepackage{hyperref}

\numberwithin{equation}{section}

\theoremstyle{plain}
\newtheorem{theorem}{Theorem}[section]
\newtheorem{lemma}[theorem]{Lemma}
\newtheorem{proposition}[theorem]{Proposition}
\newtheorem{corollary}[theorem]{Corollary}

\theoremstyle{definition}

\newcommand{\mat}{M}                      
\newcommand{\matD}{D}                     
\newcommand{\row}[1]{\mat_{#1 \ast}}      
\newcommand{\col}[1]{\mat_{\ast #1}}      
\newcommand{\rowD}[1]{\matD_{#1 \ast}}    
\newcommand{\colD}[1]{\matD_{\ast #1}}    
\newcommand{\Ckarika}{C^{\circ}}          
\newcommand{\Cpotty}{C^{\bullet}}         
\newcommand{\fpotty}{f^{\bullet}}         
\newcommand{\str}[2]{(#1,#2)}             
\newcommand{\strA}[1]{\str{A}{#1}}        
\newcommand{\extsub}[1]{\widetilde{#1}}   
\newcommand{\ext}[1]{\widehat{#1}}        
\newcommand{\genalg}[1]{[#1]}             
\newcommand{\genrcl}[1]{\langle #1 \rangle_\exists} 
\newcommand{\genwrcl}[1]{\langle #1 \rangle_\nexists} 

\newcommand{\dom}{{\operatorname{dom}}\,}
\DeclareMathOperator{\Sol}{Sol}
\DeclareMathOperator{\Pol}{Pol}
\DeclareMathOperator{\pPol}{pPol}
\DeclareMathOperator{\Inv}{Inv}
\DeclareMathOperator{\Str}{Str}
\DeclareMathOperator{\Clo}{Clo}
\DeclareMathOperator{\SPfin}{SP_{fin}}
\DeclareMathOperator{\HSP}{HSP}
\DeclareMathOperator{\h}{ht}

\newcommand{\OA}{\mathcal{O}_A}
\newcommand{\OAn}{\mathcal{O}_A^{(n)}}
\newcommand{\OAk}{\mathcal{O}_A^{(k)}}
\newcommand{\PA}{\mathcal{P}_A}

\newcommand{\PAk}{\mathcal{P}_A^{(k)}}
\newcommand{\Cn}{C^{(n)}}

\newcommand{\RA}{\mathcal{R}_A}
\newcommand{\Ank}{A^{n \times k}}
\newcommand{\pres}{\vartriangleright}
\newcommand{\sdc}{\textup{(SDC)}}
\newcommand{\rh}{\subseteq}
\newcommand{\av}{\mathbf{a}}
\newcommand{\dv}{\mathbf{d}}
\newcommand{\alg}{\mathbb{A}}
\newcommand{\Balg}{\mathbb{B}}
\newcommand{\Calg}{\mathbb{C}}
\newcommand{\N}{\mathbb{N}}
\newcommand{\Z}{\mathbb{Z}}
\newcommand{\polhom}{polymorphism\hyp{}homogeneous}

\newcommand{\MR}[1]{\href{http://www.ams.org/mathscinet-getitem?mr=#1}{MR#1}}

\author{Endre~T\'{o}th \and Tam\'{a}s~Waldhauser}
	
\title[Polymorphism-homogeneity and universal algebraic geometry]{Polymorphism-homogeneity and\\ universal algebraic geometry\thanks{This research was partially supported by the National Research, Development and Innovation Office of Hungary under grants  K115518 and K128042, by the ÚNKP-20-3-SZTE-571 New National Excellence Program of the Ministry for Innovation and Technology from the source of the National Research, Development and Innovation Fund, and by grants 
NKFIH-1279-2/2020 and TKP2021-NVA-09
of the Ministry for Innovation and Technology, Hungary.}}

\affiliation{Bolyai Institute, University of Szeged, Hungary}


\received{2020-11-13}
\revised{2021-06-22}
\accepted{2021-10-01}

\begin{document}

\publicationdetails{23}{2022}{2}{2}{6904}

\maketitle

\begin{abstract}
We assign a relational structure to any finite algebra in a canonical way, using solution sets of equations, and we prove that this relational structure is polymorphism\hyp{}homogeneous if and only if the algebra itself is polymorphism\hyp{}homogeneous.
We show that polymorphism\hyp{}homogeneity is also equivalent to the property that algebraic sets (i.e., solution sets of systems of equations) are exactly those sets of tuples that are closed under the centralizer clone of the algebra.
Furthermore, we prove that the aforementioned properties hold if and only if the algebra is injective in the category of its finite subpowers.
We also consider two additional conditions: a stronger variant for polymorphism\hyp{}homogeneity and for injectivity, and
we describe explicitly the finite semilattices, lattices, Abelian groups and monounary algebras satisfying any one of these three conditions.
\end{abstract}

\begin{center}
\emph{Dedicated to Maurice Pouzet on the occasion of his 75th birthday.}
\end{center}

\section{Introduction}\label{sect:introduction}

Various notions of homogeneity appear in several areas of mathematics, such as model theory, group theory, combinatorics, etc.
Roughly speaking, a structure $\mathcal{A}$ is said to be homogeneous if certain kinds of local morphisms (i.e., morphisms defined on ``small'' substructures of $\mathcal{A}$) extend to endomorphisms of $\mathcal{A}$. 
Specifying the kind of morphisms that are expected to be extendible, one can define many different versions of homogeneity. 
We consider a variant called polymorphism\hyp{}homogeneity introduced by C.~Pech and M.~Pech \cite{PechPech2015} that involves ``multivariable'' homomorphisms: we require extendibility of homomorphisms defined on finitely generated substructures of direct powers of $\mathcal{A}$ (see Section~\ref{subsect:polhom} for the precise definition).

We study polymorphism\hyp{}homogeneity of finite algebraic structures and of certain relational structures constructed from algebras.
Since homomorphisms depend on the term operations, not on the particular choice of basic operations, we work mainly with the clone $C=\Clo(\alg)$ of term operations of the algebraic structure $\alg = \strA{F}$ (i.e., $C$ is the clone generated by $F$; see Section~\ref{subsect:clones}).
An $n$-ary operation $f \colon A^n \to A$ can be regarded as an $(n+1)$-ary relation, called the graph of $f$, denoted by $\fpotty$ (see Section~\ref{subsect:solution sets}).
Probably the most natural way to convert $\alg$ into a relational structure is to consider the graphs of the operations of $\alg$, thus we define $\Cpotty = \{ \fpotty : f \in C\}$ to be the set of graphs of term operations of $\alg$.
We will prove that if the relational structure $\strA{\Cpotty}$ is \polhom, then the algebra $\alg$ is also \polhom, but the converse is not true in general.

To construct a relational structure that is equivalent to $\alg$ in terms of polymorphism\hyp{}homogeneity, observe that the relation $\fpotty$ is nothing else than the solution set of the equation $f(x_1,\dots,x_n)=x_{n+1}$.
We might consider more general equations where the right hand side is not necessarily a single variable: let $\Ckarika$ be the set of solution sets of equations of the form $f(x_1,\dots,x_n)=g(x_1,\dots,x_n)$, where $f,g \in C$.
It turns out that $\strA{\Ckarika}$ is the ``right'' choice for a relational counterpart of $\alg$: the algebra $\alg$ is \polhom{} if and only if the relational structure $\strA{\Ckarika}$ is \polhom.

The elements of $\Ckarika$ are solution sets of single equations, hence intersections of such sets are solution sets of systems of equations. The latter are also called algebraic sets, as they are analogues of algebraic varieties\footnote{Note that the word variety has a different meaning in universal algebra: a variety is an equationally definable class of algebras, or, equivalently, a class of algebras that is closed under homomorphic images, subalgebras and direct products.} investigated in algebraic geometry; the study of these sets can thus be regarded as universal algebraic geometry \cite{Plotkin2007}.
Motivated by the fact that a set of vectors over a field is the solution set of a system of (homogeneous) linear equations if and only if it is closed under affine linear combinations (all linear combinations), we investigated the possibility of characterizing algebraic sets by means of closure conditions in \cite{TothWaldhauser2017,TothWaldhauser2020}.
If algebraic sets over $\alg$ are exactly those sets of tuples that are closed under a suitably chosen set of operations, then we say that $\alg$ has property \sdc{} (see Section~\ref{subsect:solution sets} for an explanation).
We will see that this property is equivalent to polymorphism\hyp{}homogeneity of $\strA{\Ckarika}$ and of $\alg$.

The categorical notion of injectivity also asks for extensions of certain homomorphisms, so it is not surprising that a finite algebra $\alg$ is \polhom{} if and only if it is injective in a certain class of algebras, namely in the class of finite subpowers of $\alg$ (see Section~\ref{subsect:injective} for the definitions).
Perhaps it is more natural to consider injectivity in the variety $\HSP\alg$ generated by $\alg$, hence we will also investigate the relationship between this notion and polymorphism\hyp{}homogeneity.

Figure~\ref{fig:conditions} shows the six properties that we are concerned with in this paper.
In Section~\ref{sect:equivalences} we prove all the implications and equivalences indicated in the figure.
It turns out that for finite algebras four of the six conditions are equivalent, thus we have actually three different properties marked by the three boxes.
In Section~\ref{sect:examples} we determine finite semilattices, lattices, Abelian groups and monounary algebras possessing these three properties, and these examples will justify all of the ``non-implications'' in Figure~\ref{fig:conditions}.

\begin{figure}
\centering
\begin{tikzpicture}
[
cardinvisible/.style=
    {
    rectangle, 
    font=\small,
    text width=3.63cm,
    minimum height=1.2cm,
    inner sep=9pt,
    },
card/.style=
    {
    rectangle, 
    draw=black,
    font=\small,
    text width=3.63cm,
    minimum height=1.2cm,
    inner sep=9pt,
    },
doublearrow/.style=
    {
    line width=1pt, 
    black,
    double distance = 4pt,
    >={Implies[length=0pt 1]},
    },
strike through/.style={
    draw=gray,
    postaction=decorate,
    decoration={
      markings,
      mark=at position 0.5 with {
        \draw[-] (-6pt,-6pt) -- (6pt, 6pt);
        \draw[-] (-6pt,6pt) -- (6pt, -6pt);
      }
    }
  }
]
\matrix [column sep=1.2cm, row sep=1.2cm]
{
\node[card] (Cpotty) {$\strA{\Cpotty}$ is \polhom}; & \node[cardinvisible] (Ckarika) {$\strA{\Ckarika}$ is \polhom}; & \node[cardinvisible] (C) {$\alg$ is polymorphism-\\homogeneous}; \\
\node[card] (injHSP) {$\alg$ is injective in $\HSP(\alg)$}; & \node[cardinvisible] (injSPfin) {$\alg$ is injective in $\SPfin(\alg)$}; & \node[cardinvisible] (SDC) {$\alg$ has property \sdc}; \\
};
\draw[thin] ($(Ckarika.north west)+(0.01,0.0)$)  rectangle ($(SDC.south east)+(-0.6,-0.0)$);
\draw[<->,doublearrow] (Ckarika) -- (C);
\draw[<->,doublearrow] (injSPfin) -- (SDC);
\draw[<->,doublearrow] (injSPfin) -- (Ckarika);
\draw[<->,doublearrow] (C) -- (SDC);
\draw[->,doublearrow] ([yshift=10pt] Cpotty.east) -- ([yshift=10pt] Ckarika.west);
\draw[->,doublearrow,strike through] ([yshift=-10pt] Ckarika.west) -- ([yshift=-10pt] Cpotty.east);
\draw[->,doublearrow] ([yshift=10pt] injHSP.east) -- ([yshift=10pt] injSPfin.west);
\draw[<-,doublearrow,strike through] ([yshift=-10pt] injHSP.east) -- ([yshift=-10pt] injSPfin.west);
\draw[->,doublearrow,strike through] ([xshift=-10pt] Cpotty.south) -- ([xshift=-10pt] injHSP.north);
\draw[->,doublearrow,strike through] ([xshift=10pt] injHSP.north) -- ([xshift=10pt] Cpotty.south);
\end{tikzpicture}
\caption{Relationships between property \sdc{} and several variants of polymorphism\hyp{}homogeneity and injectivity.}
\label{fig:conditions}
\end{figure}

\section{Preliminaries}

\subsection{Clones and relational clones}\label{subsect:clones}
Let $\OAn$ denote the set of all  \emph{$n$-ary operations} on a set $A$ (i.e., maps $f \colon A^n \to A$), and let $\OA$ be the set of all operations of arbitrary finite arities on $A$.
In this paper we will always assume that the set $A$ on which we consider operations and relations is finite.
The \emph{composition} of $f\in\OAn$ by $g_1,\dots,g_n \in \OAk$ is the $k$-ary operation $f(g_1,\dots,g_n)$ defined by
\[
	f(g_1,\dots,g_n)(\av) = f(g_1(\av),\dots,g_n(\av)) \qquad (\av \in A^k).
\] 
A set $C \rh \OA$ of operations is a \emph{clone} if $C$ is closed under composition and contains the projections $(x_1,\dots,x_n) \mapsto x_i$ for $1 \leq i \leq n$.
We use the symbol $\Cn$ for the $n$-ary part of $C$, i.e., $\Cn = C \cap \OAn$.
The clone \emph{generated} by $F \rh \OA$ is the least clone $\Clo(F)$ containing $F$. 
By the definition of a term operation, $\Clo(F)$ is the clone of term operations of the algebra $\alg =\strA{F}$, hence we will also use the notation $\Clo(\alg)$ for this clone.

A \emph{$k$-ary partial operation} on $A$ is a map $h \colon \dom h \to A$, where the \emph{domain} of $h$ can be any set $\dom h \rh A^k$.
The set of all partial operations on $A$ is denoted by $\PA$, and the set of all $k$-ary partial operations on $A$ is denoted by $\PAk$.
A \emph{strong partial clone} is a set of partial operations that is closed under composition, contains the projections, and contains all restrictions of its members to arbitrary subsets of their domains.  
Note that if $C \rh \OA$ is a clone, then the least strong partial clone $\Str(C)$ containing $C$ consists of all restrictions of elements of $C$, i.e., $h \in \PA$ belongs to $\Str(C)$ if and only if $h$ can be extended to a total operation $\ext{h} \in C$.

An $n$-ary \emph{relation} on $A$ is a subset of $A^n$; the set of all relations (of arbitrary arities) on $A$ is denoted by $\RA$.
Given a set of relations $R \rh \RA$, a \emph{primitive positive formula} $\Phi(x_1,\dots,x_n)$ over $R$ is an existentially quantified conjunction:
\begin{equation}
\Phi(x_1,\dots,x_n)  = \exists y_1 \cdots \exists y_m \bigwith_{i=1}^{t}\rho_{i}\bigl(z_1^{(i)},			\dots,z_{r_i}^{(i)}\bigr) \label{eq:pp formula},
\end{equation}
where $\rho_i \in R$ is a relation of arity  $r_i$, and each $z_j^{(i)}$ is a variable from the set $\{x_1,\dots,x_n, y_1,\dots,y_m\}$ for $i=1,\dots,t, \ j=1,\dots,r_i$.
The relation $\rho = \{ (a_1, \dots, a_n) : \Phi(a_1,\dots,a_n) \text{ is true}\} \rh A^n$ is then said to be \emph{defined by the primitive positive formula} $\Phi$.
The set of all primitive positive definable relations over $R$ is denoted by $\genrcl{R}$, and such sets of relations are called \emph{relational clones}.
If we allow only quantifier-free primitive positive formulas, then we obtain the \emph{weak relational clone} $\genwrcl{R}$.
	
\subsection{Galois connections between operations and relations} \label{subsect:Galois}

If $\mat=(m_{ij}) \in \Ank$ is an $n \times k$ matrix over the set $A$, then we denote the $i$-th row and the $j$-th column of $\mat$ by $\row{i}$ and $\col{j}$, respectively:
\begin{align*}
\row{i} &= (m_{i1},\dots,m_{ik}) \qquad (i=1,\dots,n), \\
\col{j} &= (m_{1j},\dots,m_{nj}) \qquad (j=1,\dots,k).
\end{align*}
If $h \in \PAk$ is a partial operation of arity $k$ such that the rows of $\mat$ are in the domain of $h$, then we can apply $h$ to each row of $\mat$, and then we obtain the $n$-tuple $\big(h(\row{1}),\dots,h(\row{n})\big)$. We may also denote this tuple by $h(\col{1},\dots,\col{k})$, as it is nothing but the componentwise application of $h$ to the $k$ columns of $\mat$. Therefore, we have 
\[
\big(h(\row{1}),\dots,h(\row{n})\big) = h(\col{1},\dots,\col{k}).
\]
We will often use the above equality without further mention.

We say that a $k$-ary (partial) operation $h$ \emph{preserves} the relation $\rho \rh A^n$, denoted as $h \pres \rho$, if for every matrix $M \in \Ank$ such that each column of $\mat$ belongs to $\rho$ (and each row of $\mat$ is in the domain of $h$), we have $h(\col{1},\dots,\col{k}) \in  \rho$.
If $R$ is a set of relations, then we write $h \pres R$ to indicate that $h$ preserves all elements of $R$.
In other words, $h \pres R$ holds if and only if $h$ is a (partial) polymorphism of the relational structure $\mathcal{A}=\strA{R}$, i.e., $h$ is a homomorphism from (the substructure $\dom h$ of) $\mathcal{A}^k$ to $\mathcal{A}$.
The set of all (partial) operations preserving each relation of $R$ is denoted by $\Pol R$ ($\pPol R$), and the set of all relations preserved by each member of a set $F$ of (partial) operations is denoted by $\Inv F$:
\begin{align*}
 \Pol R &= \big\{ h \in \OA : h \pres \rho \text{ for every } \rho \in R \big\}; \\
\pPol R &= \big\{ h \in \PA : h \pres \rho \text{ for every } \rho \in R \big\}; \\
 \Inv F &= \big\{ \rho \in \RA : h \pres \rho \text{ for every } h \in F \big\}.
\end{align*}
Note that $\Pol R = \pPol R \cap \OA$.

The closed sets under the Galois connection $\Pol - \Inv$ ($\pPol - \Inv$) between (partial) operations and relations are exactly the (strong partial) clones and the (weak) relational clones; this makes these Galois connections fundamental tools in clone theory.

\begin{theorem}[\cite{BKKR1969,Geiger1968,Romov1981}]\label{thm:Galois}
For any set of operations $F \rh \OA$ and any set of relations $R \rh \RA$, we have $\Clo(F) = \Pol \Inv F$ and $\genrcl{R} = \Inv \Pol R$.
For any set of partial operations $F \rh \PA$ and any set of relations $R \rh \RA$, we have $\Str(F) = \pPol \Inv F$ and $\genwrcl{R} = \Inv \pPol R$.
\end{theorem}

\subsection{Universal algebraic geometry and centralizers}\label{subsect:solution sets}

Let $\alg$ be a finite algebra and let $C = \Clo(\alg)$
be the clone of term operations of $\alg$.
If $f$ and $g$ are $n$-ary term operations of $\alg$, then $f(x_1,\dots,x_n)=g(x_1,\dots,x_n)$ is an \emph{equation} in $n$ variables over $\alg$, which we may simply write as a pair $(f,g)$.
The \emph{solution set} of $(f,g)$ is then the set $\Sol(f,g) = \{ (a_1,\dots,a_n) \in A^n : f(a_1,\dots,a_n)=g(a_1,\dots,a_n) \}$.
Of special interest are the equations of the form $f(x_1,\dots,x_n)=x_{n+1}$; the solution set of this equation is the $(n+1)$-ary relation $\fpotty = \{ (a_1,\dots,a_n,a_{n+1}) \in A^{n+1} : f(a_1,\dots,a_n)=a_{n+1}) \}$, which is called the \emph{graph} of $f$.
We use the symbols $\Cpotty$ and $\Ckarika$ for the set of graphs and for the set of all solution sets of equations over $C$:
\begin{align*}
\Cpotty  &= \big\{ \fpotty : f \in C \big\}; \\
\Ckarika &= \big\{ \Sol(f,g) : f,g \in \Cn,  n \in \N \big\}.
\end{align*}
Note that $\Cpotty \rh \Ckarika$, and it is easy to verify that $\genrcl{\Cpotty} = \genrcl{\Ckarika}$ (see Lemma~3.2 of \cite{TothWaldhauser2020}), but in general $\genwrcl{\Cpotty}$ and $\genwrcl{\Ckarika}$ may be different weak relational clones.

The members of $\genwrcl{\Ckarika}$ are intersections of solution sets of finitely many equations, i.e., $\genwrcl{\Ckarika}$ consists of solution sets of finite systems of equations over $\alg$.
Allowing infinite systems of equations, we obtain the so-called \emph{algebraic sets}, which are the main objects of study in universal algebraic geometry \cite{Plotkin2007}.
Since we deal only with finite algebras, every system of equations is equivalent to a finite system of equations, thus the elements of $\genwrcl{\Ckarika}$ are exactly the algebraic sets.

As mentioned in Section~\ref{sect:introduction}, basic results of linear algebra hint at the possibility that algebraic sets can sometimes be described by closure conditions.
It turns out that if there is a clone $D$ such that algebraic sets are exactly those sets of tuples that are closed under $D$, then $D$ must be the clone $C^{\ast} = \Pol \Cpotty$ (see Corollary~3.7 in \cite{TothWaldhauser2020}).
This clone is called the \emph{centralizer} of $C$, since it consists of those operations that commute with every member of $C$; in other words, a $k$-ary operation $h$ belongs to $C^{\ast}$ if and only if $h$ is a homomorphism from $\alg^k$ to $\alg$.
(Observe that since $\genrcl{\Cpotty} = \genrcl{\Ckarika}$, the centralizer can equivalently be defined as $C^{\ast} = \Pol \Ckarika$, by Theorem~\ref{thm:Galois}.)

If the algebraic sets (i.e., solution sets of systems of equations) of $\alg$ coincide with the $C^{\ast}$-closed sets of tuples, then we say that the algebra $\alg$ has property \sdc{}; this abbreviation stands for ``Solution sets are Definable by closure under the Centralizer''.
We proved in \cite{TothWaldhauser2017} that every two-element algebra has this property, and in \cite{TothWaldhauser2020} finite semilattices and lattices with property \sdc{} were characterized (see Sections~\ref{subsect:semilattices} and \ref{subsect:lattices}).
In general, property \sdc{} is easily seen to be equivalent to the condition $\genrcl{\Ckarika} = \genwrcl{\Ckarika}$, i.e., the algebra $\alg$ has property \sdc{} if and only if quantifiers can be eliminated from primitive positive formulas over $\Ckarika$. Let us state this fact explicitly for later reference together with a few more equivalent conditions  (see Theorem~3.6 of \cite{TothWaldhauser2020}).

\begin{theorem}[\cite{TothWaldhauser2020}] \label{thm:3.6}
	For every clone $C$ on a finite set $A$, the following five conditions are equivalent:
	\begin{enumerate}[label=\normalfont(\roman*)]
		\item $C$ has Property~$\sdc$;
		\item $\langle C^\circ\rangle _\nexists=\Inv{(C^{\ast})}$;
		\item $\langle C^\circ\rangle_\nexists=\langle C^\circ\rangle_\exists$;
		\item every primitive positive formula over $C^{\circ}$ is equivalent to a quantifier-free primitive positive formula over $C^\circ$;
		\item $\langle C^\circ\rangle _\nexists$ is a relational clone.
	\end{enumerate}
\end{theorem}

\subsection{Polymorphism\hyp{}homogeneity}\label{subsect:polhom}

A first-order structure $\mathcal{A}$ (i.e., a set $A$ equipped with relations and/or operations) is said to be 
\emph{$k$-polymor\-phism\hyp{}homogeneous}, 
if every homomorphism $h\colon \mathcal{B} \to \mathcal{A}$ defined on a finitely generated substructure $\mathcal{B} \leq \mathcal{A}^k$ extends to a homomorphism $\ext{h}\colon \mathcal{A}^k \to \mathcal{A}$.
(Considering only finite structures, the assumption that $\mathcal{B}$ is finitely generated can be omitted from the definition.)
The case $k=1$ gives the notion of \emph{homomorphism\hyp{}homogeneity} introduced by P.~J.~Cameron and J.~Ne\v{s}et\v{r}il \cite{CameronNesetril2006}.
If $\mathcal{A}$ is $k$-\polhom{} for every natural number $k$, then we say that $\mathcal{A}$ is \emph{\polhom} \cite{PechPech2015}.
These two notions are linked by the following result, which was proved for relational structures by C.~Pech and M.~Pech \cite{PechPech2015} and for algebraic structures by Z.~Farkasov\'{a} and D.~Ja\-ku\-b\'{\i}\-ko\-v\'{a}\hyp{}Stu\-de\-novs\-k\'{a} \cite{FarkasovaJakubikova2015}, but the same proof works for arbitrary first-order structures.

\begin{proposition}[\cite{FarkasovaJakubikova2015,PechPech2015}]\label{prop:direct powers}
A first-order structure $\mathcal{A}$ is \polhom{} if and only if $\mathcal{A}^k$ is homomorphism\hyp{}homogeneous for all natural numbers $k$.
\end{proposition}

In the next proposition we recall a useful result from \cite{PechPech2015} that relates polymorphism\hyp{}homogeneity and quantifier elimination for finite relational structures; we give a short proof utilizing the Galois connections between (partial) operations and relations.

\begin{proposition}[\cite{PechPech2015}]\label{prop:QE <==> pol-hom}
A finite relational structure has quantifier elimination for primitive positive formulas if and only if it is \polhom.
\end{proposition}
\begin{proof}
A finite relational structure $\mathcal{A}=\strA{R}$ has quantifier elimination for primitive positive formulas if and only if $\genwrcl{R}=\genrcl{R}$.
Using the Galois connections $\Pol - \Inv$ (clones and relational clones) and $\pPol - \Inv$ (strong partial clones and weak relational clones), we can reformulate this condition in several steps to reach polymorphism\hyp{}homogeneity:
\begin{align*}
      \genwrcl{R}=\genrcl{R}
&\iff \Inv\pPol R = \Inv\Pol R \\
&\iff \pPol\Inv\pPol R = \pPol\Inv\Pol R \\
&\iff \pPol R = \Str(\Pol R) \\
&\iff  \{ h \in \PA : h \pres R \} = \{ h \in \PA : h \text{ extends to } \ext{h} \in \OA \text{ such that } \ext{h} \pres R\} \\
&\iff \mathcal{A} \text{ is \polhom.}\qedhere
\end{align*}
\end{proof}

\subsection{Injectivity}\label{subsect:injective}

Let $\mathcal{K}$ be a class of algebras and $\alg \in \mathcal{K}$.
We say that $\alg$ is \emph{injective} in $\mathcal{K}$ if every homomorphism $h\colon \Balg \to \alg$ extends to a homomorphism $\ext{h} \colon \Calg \to \alg$ whenever $\Balg,\Calg \in \mathcal{K}$ and $\Balg \leq \Calg$.
Clearly, if $\alg$ is injective in $\mathcal{K}$, then $\alg$ is also injective in every subclass of $\mathcal{K}$ that contains $\alg$.
Injectivity is most often considered in the largest relevant class $\mathcal{K}$; for example, if $\alg$ is a group or a lattice, then $\mathcal{K}$ is usually chosen to be the class of all groups or lattices.
In this paper we shall consider smaller classes, namely the variety generated by $\alg$ and the set of finite subpowers of $\alg$. Recall that a class of algebras is a variety if and only if it is closed under direct products, subalgebras, and homomorphic images. The variety generated by $\alg$, denoted by $\HSP\alg$, consists of the homomorphic images of subalgebras of direct powers of $\alg$. Analogously, $\SPfin\alg$ denotes the class of subalgebras of finite direct powers of $\alg$, and we refer to algebras in this class as finite subpowers of $\alg$. 
Let us mention that in \cite{KovacsNewman1982} a group $\alg$ is called \emph{relatively injective} if it is injective in the variety $\HSP\alg$.

\section{Polymorphism\hyp{}homogeneity, algebraic sets and injectivity}\label{sect:equivalences}

First let us prove the equivalences shown on the right hand side of Figure~\ref{fig:conditions}.
The equivalence of property \sdc{} and polymorphism\hyp{}homogeneity of $\strA{\Ckarika}$ follows immediately from Proposition~\ref{prop:QE <==> pol-hom}.

\begin{proposition}\label{prop:SDC <==> pol-hom rel}
If $\alg$ is a finite algebra and $C=\Clo(\alg)$, then  $\alg$ has property \sdc{} if and only if $\strA{\Ckarika}$ is \polhom.
\end{proposition}
\begin{proof}
By Theorem~\ref{thm:3.6}, property \sdc{} of $\alg$ is equivalent to quantifier elimination for primitive positive formulas for the relational structure $\strA{\Ckarika}$, and the latter is equivalent to polymorphism\hyp{}homogeneity of $\strA{\Ckarika}$ by Proposition~\ref{prop:QE <==> pol-hom}.
\end{proof}

In the next theorem we establish the connection between ``algebraic'' and ``relational'' polymorphism\hyp{}homogeneity.
We need two technical lemmas for the proof of this result.

\begin{lemma}\label{lem:dom subalg}
Suppose that $\alg$ is a finite algebra, $C=\Clo(\alg)$ and  $h \in \PAk$ is a $k$-ary partial operation on $A$. 
If $\dom h$ is a subalgebra of $\alg^k$, then the following three conditions are equivalent:
\begin{enumerate}[label=(\alph*)]
\item $h \pres \Cpotty$;
\item $h \pres \Ckarika$;
\item $h$ is a homomorphism from $\dom h$ to $\alg$.
\end{enumerate}
\end{lemma}
\begin{proof}
To show that (a) implies (b), assume that $h \pres \Cpotty$, and let $\rho = \Sol(f,g)$, where $f,g \in \Cn$; we shall prove that $h \pres \rho$.
Let $\mat \in \Ank$ be a matrix such that each row of $\mat$ belongs to $\dom h$ and each column of $\mat$ belongs to $\rho$.
Then we have $f(\row{1},\dots,\row{n}) \in \dom h$, as $\dom h$ is a subalgebra of $\alg^k$ and $f \in C$.
Now let $\mat'\in A^{(n+1)\times k}$ be the matrix obtained by adding the row $f(\row{1},\dots,\row{n})$ to the bottom of $\mat$.
Since $f(\row{1},\dots,\row{n})=(f(\col{1}),\dots,f(\col{k}))$, every column of $\mat'$ belongs to $\fpotty$, hence applying $h$ to each row of $\mat'$, we obtain a tuple in $\fpotty$, because $h$ preserves $\fpotty$ by our assumption. 
This means that 
\begin{equation}\label{eq:hf in lem:dom subalg}
h(f(\col{1}),\dots,f(\col{k}))=f(h(\col{1},\dots,\col{k})).    
\end{equation}
Using a similar argument, replacing $f$ by $g$, we obtain
\begin{equation}\label{eq:hg in lem:dom subalg}
h(g(\col{1}),\dots,g(\col{k}))=g(h(\col{1},\dots,\col{k})).
\end{equation}
All columns of $M$ were assumed to be in the relation $\rho = \Sol(f,g)$; therefore,
\begin{equation}\label{eq:fg in lem:dom subalg}
f(\col{1})=g(\col{1}),\dots,f(\col{k})=g(\col{k}).
\end{equation}
Combining \eqref{eq:hf in lem:dom subalg}, \eqref{eq:hg in lem:dom subalg} and \eqref{eq:fg in lem:dom subalg}, we can conclude that $f(h(\col{1},\dots,\col{k}))=g(h(\col{1},\dots,\col{k}))$, hence $h(\col{1},\dots,\col{k}) \in \rho$, and this proves that $h \pres \rho$.

Next suppose that (b) holds; to prove that $h$ is an algebra homomorphism, consider an operation $f \in \Cn$ and tuples $\dv_1,\dots,\dv_n \in \dom h$.
Since $\dom h$ is a subalgebra, we have $f(\dv_1,\dots,\dv_n) \in \dom h$.
Let $\mat \in A^{(n+1) \times k}$ be the matrix whose rows are $\dv_1,\dots,\dv_n, f(\dv_1,\dots,\dv_n)$.
Then all columns of $\mat$ belong to $\fpotty$, hence $(h(\dv_1),\dots,h(\dv_n),h(f(\dv_1,\dots,\dv_n))) \in \fpotty$, since $h$ was assumed to preserve $\Ckarika$ (recall that $\Ckarika \supseteq \Cpotty$).
Thus we have $f(h(\dv_1),\dots,h(\dv_n))=h(f(\dv_1,\dots,\dv_n))$, proving that $h$ is indeed a homomorphism.

Finally, assume (c) and let us verify (a).
Let $f \in \Cn$, and let $\mat \in A^{(n+1) \times k}$ be an arbitrary  matrix whose rows and columns belong to $\dom h$ and to $\fpotty$, respectively; in particular, the last row of $\mat$ is $f(\row{1},\dots,\row{n})$.
We need to show that $h(\col{1},\dots,\col{k}) \in \fpotty$, which is equivalent to $f(h(\row{1}),\dots,h(\row{n}))=h(f(\row{1},\dots,\row{n}))$.
The latter equality is justified by the fact that $h$ is a homomorphism.
\end{proof}

\begin{lemma}\label{lem:extension}
Let $C$ be a clone on a finite set $A$, and let $h \in \PAk$ be a $k$-ary partial operation on $A$. 
If $h$ preserves $\Ckarika$, then $h$ can be extended to a partial operation $\extsub{h} \in \PAk$ such that  $\extsub{h} \pres \Ckarika$ and $\dom\extsub{h}=\genalg{\dom h}$ (the subalgebra of $\alg^k$ generated by $\dom h$).
\end{lemma}
\begin{proof}
If $\av \in \genalg{\dom h}$, then $\av$ can be obtained from the elements of $\dom h$ by an operation $t\in C$. Adding inessential variables to $t$ if necessary, we can assume that actually all elements of $\dom h$ are used, and thus the arity of $t$ is $m:=|\dom h|$.
Therefore, we can write $\av = t(\dv_1,\dots,\dv_m)$, where $\dom h = \{ \dv_1,\dots,\dv_m \}$ and $t \in C^{(m)}$. 
We then define the desired extension of $h$ at $\av$ by
\begin{equation}\label{eq:extension}
\extsub{h}(\av)=t(h(\dv_1),\dots,h(\dv_m)).    
\end{equation}

First we need to verify that $\extsub{h}$ is well defined.
Suppose that an element $\av \in \genalg{\dom h}$ can be written in more than one way in the above form: $\av=t_1(\dv_1,\dots,\dv_m)=t_2(\dv_1,\dots,\dv_m)$ with $t_1,t_2 \in C^{(m)}$.
Setting $\rho = \Sol(t_1,t_2)  \in \Ckarika$, and letting $\matD \in A^{m \times k}$ be the matrix with rows $\dv_1,\dots,\dv_m$, every column of $\matD$ belongs to the relation $\rho$.
Since $h$ preserves $\rho$, we have $h(\colD{1},\dots,\colD{k}) \in \rho$, and therefore $t_1(h(\colD{1},\dots,\colD{k}))=t_2(h(\colD{1},\dots,\colD{k}))$ holds. This implies that $\extsub{h}$ is well defined, as the value of $\extsub{h}(\av)$ in \eqref{eq:extension} does not depend on the particular choice of the operation $t$:
\begin{align*}
t_1(h(\dv_1),\dots,h(\dv_m))&=
t_1(h(\rowD{1}),\dots,h(\rowD{m}))\\
&=t_1(h(\colD{1},\dots,\colD{k}))\\
&=t_2(h(\colD{1},\dots,\colD{k}))\\
&=t_2(h(\rowD{1}),\dots,h(\rowD{m}))\\
&=t_2(h(\dv_1),\dots,h(\dv_m)).
\end{align*}

Choosing the $i$-th projection $t(x_1,\dots,x_m)=x_i$ in \eqref{eq:extension}, we see that $\extsub{h}(\dv_i)=h(\dv_i)$ for all $i \in \{ 1,\ldots,m \}$, thus $\extsub{h}$ is an extension of $h$. It remains to prove that $\extsub{h}$ preserves $\Ckarika$. 

Let $f,g \in \Cn$, let $\rho = \Sol(f,g) \in \Ckarika$, and let $\mat \in \Ank$ be a matrix such that all rows of $\mat$ are in $\dom \extsub{h}$ and each column of $\mat$ belongs to the relation $\rho$.
Since $\row{i} \in \dom \extsub{h} = \genalg{\dom h}$, we can write $\row{i}=t_i(\dv_1,\dots,\dv_m)$ for suitable operations $t_i \in C^{(m)}$ for $i=1,\dots,n$.
Since $\col{j} \in \rho$, we have $f(\col{j})=g(\col{j})$ for $j=1,\dots,k$, or, equivalently, $f(\row{1},\dots,\row{n})=g(\row{1},\dots,\row{n})$.
Combining the latter two observations, we get that
\begin{align*}
 f(t_1,\dots,t_n)(\dv_1,\dots,\dv_m)
 &=f(t_1(\dv_1,\dots,\dv_m),\dots,t_n(\dv_1,\dots,\dv_m))\\
 &=f(\row{1},\dots,\row{n})\\
 &=g(\row{1},\dots,\row{n})\\
 &=g(t_1(\dv_1,\dots,\dv_m),\dots,t_n(\dv_1,\dots,\dv_m))\\
 &=g(t_1,\dots,t_n)(\dv_1,\dots,\dv_m).
\end{align*}
Using the notation
$f'=f(t_1,\dots,t_n)$ and $g'=g(t_1,\dots,t_n)$, we can summarize the calculation above as $f'(\dv_1,\dots,\dv_m)=g'(\dv_1,\dots,\dv_m)$, which means that the columns of the matrix $D$ belong to the relation $\rho':=\Sol(f',g')$.
The clone $C$ is closed under composition, thus $f',g' \in C$, hence $\rho' \in \Ckarika$.
We assumed that $h \pres \Ckarika$; therefore, we have $(h(\dv_1),\dots,h(\dv_m)) \in \rho'$, and this is equivalent to $f'(h(\dv_1),\dots,h(\dv_m))=g'(h(\dv_1),\dots,h(\dv_m))$.
Expanding this last equality using the definition of $f'$ and $g'$ together with \eqref{eq:extension}, we obtain $\extsub{h}(\col{1},\dots,\col{k})) \in \rho$, which completes the proof of $\extsub{h} \pres \Ckarika$:
\begin{align*}
 f(\extsub{h}(\col{1},\dots,\col{k}))
 &=f(\extsub{h}(\row{1}),\dots,\extsub{h}(\row{n}))\\
 &=f\big(\extsub{h}(t_1(\dv_1,\dots,\dv_m)),\dots,\extsub{h}(t_n(\dv_1,\dots,\dv_m))\big)\\
 &=f\big(t_1(h(\dv_1),\dots,h(\dv_m)),\dots,t_n(h(\dv_1),\dots,h(\dv_m))\big)\\
 &=f(t_1,\dots,t_n)(h(\dv_1),\dots,h(\dv_m))\\
 &=f'(h(\dv_1),\dots,h(\dv_m))\\
 &=g'(h(\dv_1),\dots,h(\dv_m))\\
 &=g(t_1,\dots,t_n)(h(\dv_1),\dots,h(\dv_m))\\
 &=g\big(t_1(h(\dv_1),\dots,h(\dv_m)),\dots,t_n(h(\dv_1),\dots,h(\dv_m))\big)\\
 &=g\big(\extsub{h}(t_1(\dv_1,\dots,\dv_m)),\dots,\extsub{h}(t_n(\dv_1,\dots,\dv_m))\big)\\
 &=g(\extsub{h}(\row{1}),\dots,\extsub{h}(\row{n}))\\
 &= g(\extsub{h}(\col{1},\dots,\col{k})).\qedhere
\end{align*}
\end{proof}

\begin{theorem}\label{thm:pol-hom Ckarika <==> pol-hom alg}
If $\alg$ is a finite algebra and $C=\Clo(\alg)$, then  $\alg$ is \polhom{} if and only if $\strA{\Ckarika}$ is \polhom.
\end{theorem}
\begin{proof}
Assume first that $\alg$ is \polhom{}, and consider an arbitrary partial polymorphism $h$ of $\strA{\Ckarika}$, i.e., $h \in \PAk$ preserves $\Ckarika$.
By Lemma~\ref{lem:extension}, $h$ can be extended to a $\Ckarika$-preserving partial operation $\extsub{h}$ defined on the subalgebra $\genalg{\dom h} \leq \alg^k$. 
Applying Lemma~\ref{lem:dom subalg} to $\extsub{h}$, we see that $\extsub{h}$ is an algebra homomorphism from $\genalg{\dom h}$ to $\alg$.
Since $\alg$ is \polhom{}, $\extsub{h}$ extends to a homomorphism $\ext{\extsub{h}}\colon \alg^k \to \alg$.
Using Lemma~\ref{lem:dom subalg} again, we see that $\ext{\extsub{h}}$ preserves $\Ckarika$, hence it is a polymorphism of the relational structure $\strA{\Ckarika}$, and this proves that the latter is \polhom{}.

Now suppose that $\strA{\Ckarika}$ is \polhom, and let $h \in \PAk$ be a homomorphism from a subalgebra $\dom h \leq \alg^k$ to $\alg$.
Lemma~\ref{lem:dom subalg} shows that $h \pres \Ckarika$, i.e., $h$ is a partial polymorphism of $\strA{\Ckarika}$.
Since $\strA{\Ckarika}$ is \polhom{}, $h$ can be extended to a polymorphism $\ext{h}$ of $\strA{\Ckarika}$.
By Lemma~\ref{lem:dom subalg}, $\ext{h}\colon \alg^k \to \alg$ is a homomorphism, and this proves that $\alg$ is \polhom{}.
\end{proof}

To complete the proof of the equivalences in the box on the right hand side of Figure~\ref{fig:conditions}, we relate injectivity and polymorphism\hyp{}homogeneity.

\begin{proposition}\label{prop:pol-hom alg <==> injective}
If $\alg$ is a finite algebra, then $\alg$ is \polhom{} if and only if $\alg$ is injective in $\SPfin(\alg)$.
\end{proposition}
\begin{proof}
Assume that $\alg$ is \polhom, and let $\Balg,\Calg \in  \SPfin(\alg)$ such that $\Balg \leq \Calg$.
Then we have $\Balg \leq \Calg \leq \alg^k$ for some $k \in \N$; in particular, $\Balg$ is a subalgebra of $\alg^k$.
Therefore, if $h \colon \Balg \to \alg$ is a homomorphism, then $h$ extends to a homomorphism $\ext{h} \colon \alg^k \to \alg$ by the polymorphism\hyp{}homogeneity of $\alg$.
A restriction of $\ext{h}$ then gives a homomorphism form $\Calg$ to $\alg$ that extends $h$, thereby proving the injectivity of $\alg$.

Conversely, if $\alg$ is injective in $\SPfin(\alg)$ and $h \in \PAk$ is a homomorphism from a subalgebra $\dom h \leq \alg^k$ to $\alg$, then the injectivity of $\alg$ immediately yields an extension $\ext{h} \colon \alg^k \to \alg$ of $h$, thus $\alg$ is indeed \polhom.
\end{proof}

\begin{corollary}\label{cor:equivalences}
If $\alg$ is a finite algebra and $C=\Clo(\alg)$, then the following conditions are equivalent:
\begin{enumerate}[label=(\roman*)]
\item $\alg$ has property \sdc;
\item $\alg$ is \polhom;
\item $\strA{\Ckarika}$ is \polhom;
\item $\alg$ is injective in $\SPfin(\alg)$.
\end{enumerate}
\end{corollary}
\begin{proof}
Combine propositions~\ref{prop:SDC <==> pol-hom rel} and \ref{prop:pol-hom alg <==> injective} and Theorem~\ref{thm:pol-hom Ckarika <==> pol-hom alg}.
\end{proof}

It remains to verify the ``one-way'' implications in Figure~\ref{fig:conditions}. Since $\HSP(\alg) \supseteq \SPfin(\alg)$, it is trivial that if $\alg$ is injective in $\HSP(\alg)$, then it is also injective in $\SPfin(\alg)$.
We end this section by proving the remaining implication; in fact, we formulate it in a bit more explicit form, which will be useful in the next section.

\begin{proposition}\label{prop:pol-hom Cpotty ==> pol-hom Ckarika}
If $\alg$ is a finite algebra and $C=\Clo(\alg)$, then $\strA{\Cpotty}$ is \polhom{} if and only if $\strA{\Ckarika}$ is \polhom{} and $\genwrcl{\Cpotty} = \genwrcl{\Ckarika}$.
\end{proposition} 
\begin{proof}
According to Proposition~\ref{prop:QE <==> pol-hom}, we need to prove the following equivalence:
\begin{equation*}
\genwrcl{\Cpotty} = \genrcl{\Cpotty} \iff \genwrcl{\Ckarika} = \genrcl{\Ckarika} \text{ and } \genwrcl{\Cpotty} = \genwrcl{\Ckarika}.
\end{equation*}
This follows immediately from the following chain of containments (the last containment is Lemma~3.2 of \cite{TothWaldhauser2020}, the others are trivial):
\begin{equation*}
\genwrcl{\Cpotty} \rh \genwrcl{\Ckarika} \rh \genrcl{\Ckarika} =  \genrcl{\Cpotty}. \qedhere
\end{equation*}
\end{proof}

\section{Examples} \label{sect:examples}

We describe explicitly the finite algebras satisfying the properties considered in the previous section in certain well known varieties: semilattices, lattices, Abelian groups and monounary algebras.
These characterizations will provide counterexamples showing that the only valid implications among these properties are the ones shown in Figure~\ref{fig:conditions}.

\subsection{Semilattices} \label{subsect:semilattices}

If we consider finite semilattices, then it turns out that five of the six conditions of Figure~\ref{fig:conditions} are equivalent, and these semilattices have already been determined in the literature \cite{BrunsLakser1970,HornKimura1971,TothWaldhauser2020}.

\begin{theorem} \label{thm:semilattices}
If $\alg$ is a finite semilattice and $C=\Clo(\alg)$, then the following conditions are equivalent:
\begin{enumerate}[label=(\roman*)]
\item $\alg$ has property \sdc;
\item $\alg$ is \polhom;
\item $\strA{\Ckarika}$ is \polhom;
\item $\alg$ is injective in $\SPfin(\alg)$;
\item $\alg$ is injective in $\HSP(\alg)$;
\item $\alg$ is the semilattice reduct of a finite distributive lattice.
\end{enumerate}
\end{theorem}

\begin{proof}
We know that conditions (i)--(iv) are equivalent (see Corollary~\ref{cor:equivalences}), and it was proved in Theorem~5.5 of \cite{TothWaldhauser2020} that (i) is equivalent to (vi).
G.~Bruns and H.~Lakser \cite{BrunsLakser1970} and, independently, A.~Horn and N.~Kimura \cite{HornKimura1971} showed that the injective objects in the category of semilattices are the semilattice reducts of completely distributive lattices.
Therefore, if $\alg$ is the semilattice reduct of a finite distributive lattice, then $\alg$ is injective in the variety of all semilattices, thus $\alg$ is also injective in $\HSP(\alg)$.
This proves that (vi) implies (v), and taking into account that (v) obviously implies (iv), the proof is complete.
\end{proof}

The top left condition of Figure~\ref{fig:conditions}, i.e., polymorphism\hyp{}homogeneity of $\strA{\Cpotty}$, is not equivalent to the others; in fact, it is satisfied only by the trivial semilattices.
 
\begin{lemma}\label{lem:01 semilattice not polhom}
Let $\alg$ be a two-element semilattice and let $C=\Clo(\alg)$. Then the relational structure $\strA{\Cpotty}$ is not polymorphism\hyp{}homogeneous.
\end{lemma}

\begin{proof}
We can assume without loss of generality that $\alg = \str{\{0,1\}}{\land}$ with the usual ordering $0<1$.
Let us consider the equation $x \land y \land z = x \land y$.
Obviously, the solution set $S=\{0,1\}^3 \setminus \{(1,1,0)\}$ of this equation is defined by a quantifier-free primitive positive formula over $C^\circ$. 
If $S$ can be defined by a quantifier-free primitive positive formula $\Phi$ over $C^\bullet$, then we can assume without loss of generality that $\Phi$ consists of a single equality, as $S$ misses only one element of $\{0,1\}^3$ (in other words, $S$ is meet-irreducible in the lattice of subsets of $\{0,1\}^3$).
Thus $\Phi$ is of the form $f(x_i,x_j,x_k)=x_\ell$ with $i,j,k,\ell \in \{ 1,2,3 \}$. Taking associativity, commutativity, and idempotence of the operation $\land$ into account, we have the following possibilites for $\Phi$ (excluding those trivial cases where the left hand side coincides with the right hand side):
\begin{equation*}
\arraycolsep=10pt
\renewcommand{\arraystretch}{1.5}
\begin{array}{ccccc}
x=y, & x = x \land y, & x = x \land z, & x = y \land z, & x = x \land y \land z, \\
y=z, & y = x \land y, & y = x \land z, & y = y \land z, & y = x \land y \land z, \\
z=x, & z = x \land y, & z = x \land z, & z = y \land z, & z = x \land y \land z.
\end{array}
\end{equation*}
It is easy to check that $S$ does not satisfy any of the equalities above; therefore, $S$ cannot be defined by a quantifier-free primitive positive formula over $C^\bullet$.
Thus $S$ belongs to $\genwrcl{\Ckarika}$ but not to $\genwrcl{\Cpotty}$,
hence $\strA{\Cpotty}$ is not \polhom{} by Proposition~\ref{prop:pol-hom Cpotty ==> pol-hom Ckarika}.
\end{proof}
 
\begin{theorem}
If $\alg$ is a nontrivial finite semilattice and $C=\Clo(\alg)$, then the relational structure $\strA{\Cpotty}$ is not polymorphism\hyp{}homogeneous.
\end{theorem}

\begin{proof}
Let $a,b \in A$ such that $a < b$, and let us consider the same equation as in the proof of Lemma~\ref{lem:01 semilattice not polhom}. Now for the solution set $S$ of this equation we have that $S \cap \{a,b\}^3 = \{a,b\}^3 \setminus \{(b,b,a)\}$. The same argument as in the proof of Lemma~\ref{lem:01 semilattice not polhom} shows that $S$ cannot be defined by a quantifier-free primitive positive formula over $\Cpotty$.
\end{proof}

\subsection{Lattices} \label{subsect:lattices}

For finite lattices the situation is very similar to the case of semilattices: five of the six conditions of Figure~\ref{fig:conditions} are equivalent, and the sixth one is satisfied only by trivial lattices.

\begin{theorem}
If $\alg$ is a finite lattice and $C=\Clo(\alg)$, then the following conditions are equivalent:
\begin{enumerate}[label=(\roman*)]
\item $\alg$ has property \sdc;
\item $\alg$ is \polhom;
\item $\strA{\Ckarika}$ is \polhom;
\item $\alg$ is injective in $\SPfin(\alg)$;
\item $\alg$ is injective in $\HSP(\alg)$;
\item $\alg$ is a finite Boolean lattice (i.e., a direct power of the two-element chain).
\end{enumerate}
\end{theorem}

\begin{proof}
Just as in the proof of Theorem~\ref{thm:semilattices}, the equivalence of (i)--(iv) follows from Corollary~\ref{cor:equivalences}, and the equivalence of (i) and (vi) is Theorem~4.8 of \cite{TothWaldhauser2020}.
(Let us mention that I.~Dolinka and D.~Ma\v{s}ulovi\'{c} \cite{DolinkaMasulovic2011} proved that a finite lattice is homomorphism\hyp{}homogeneous if and only if it is a chain or a Boolean lattice.
This together with Proposition~\ref{prop:direct powers} can also be used to prove that (ii) and (vi) are equivalent.)
To complete the proof, it suffices to prove that (vi) implies (v).
This follows immediately from a result of R.~Balbes \cite{Balbes1967}: the injective objects in the category of distributive lattices are the complete Boolean lattices (observe that if $\alg$ is a nontrivial Boolean lattice, then $\HSP(\alg)$ is the variety of distributive lattices).
\end{proof}

\begin{lemma}\label{lem:01 lattice not polhom}
Let $\alg$ be a two-element lattice and let $C=\Clo(\alg)$. Then the relational structure $\strA{\Cpotty}$ is not polymorphism\hyp{}homogeneous.
\end{lemma}

\begin{proof}
We can assume without loss of generality that $\alg = \str{\{0,1\}}{\lor,\land}$ with the usual ordering $0<1$.
Let us consider the equation $(x_1 \lor x_2) \land (x_3 \land x_4) = x_3 \land x_4$; the solution set $S=\{0,1\}^4 \setminus \{(0,0,1,1)\}$ of this equation is defined by a quantifier-free primitive positive formula over $C^\circ$.
If $S$ can be defined by a quantifier-free primitive positive formula $\Phi$ over $C^\bullet$, then we can assume without loss of generality that $\Phi$ consists of a single equality, as $S$ misses only one element of $\{0,1\}^4$ (in other words, $S$ is meet-irreducible in the lattice of subsets of $\{0,1\}^4$).
Thus $S$ is the solution set of an equation of the form $f(x_1,x_2,x_3,x_4) = u$, where $u \in \{x_1,x_2,x_3,x_4\}$.
Note that since $f$ is generated by the lattice operations $\lor$ and $\land$, it is a monotone function.
We consider four cases corresponding to the variable $u$.
\begin{enumerate}
\item If $u=x_1$, then $f(x_1,x_2,x_3,x_4)=x_1$ holds for all $(x_1,x_2,x_3,x_4) \in S$  and $f(0,0,1,1)=1$.
%
\\ 
In particular, we have $f(0,1,1,1)=0 < 1=f(0,0,1,1)$, contradicting the monotonicity of $f$. 
\item If $u=x_2$, then $f(x_1,x_2,x_3,x_4)=x_2$ holds for all $(x_1,x_2,x_3,x_4) \in S$  and $f(0,0,1,1)=1$. 
%
\\ 
In particular, we have $f(1,0,1,1)=0 < 1=f(0,0,1,1)$, contradicting the monotonicity of $f$. 
\item If $u=x_3$, then $f(x_1,x_2,x_3,x_4)=x_3$ holds for all $(x_1,x_2,x_3,x_4) \in S$  and $f(0,0,1,1)=0$. 
%
\\ 
In particular, we have $f(0,0,1,0)=1 > 0=f(0,0,1,1)$, contradicting the monotonicity of $f$. 
\item If $u=x_4$, then $f(x_1,x_2,x_3,x_4)=x_4$ holds for all $(x_1,x_2,x_3,x_4) \in S$  and $f(0,0,1,1)=0$. 
%
\\ 
In particular, we have $f(0,0,0,1)=1>0=f(0,0,1,1)$, contradicting the monotonicity of $f$. 
\end{enumerate}
We see that $S$ cannot be defined by a quantifier-free primitive positive formula $\Phi$ over $C^\bullet$, hence $\genwrcl{\Ckarika} \neq \genwrcl{\Cpotty}$, and thus $\strA{\Cpotty}$ is not \polhom{} by Proposition~\ref{prop:pol-hom Cpotty ==> pol-hom Ckarika}.
\end{proof}
 
\begin{theorem}
If $\alg$ is a nontrivial finite lattice and $C=\Clo(\alg)$, then the relational structure $\strA{\Cpotty}$ is not polymorphism\hyp{}homogeneous.
\end{theorem}

\begin{proof}
Let $a,b \in A$ such that $a < b$, and let us consider the same equation as in the proof of Lemma~\ref{lem:01 lattice not polhom}. 
Now for the solution set $S$ of this equation we have that $S \cap \{a,b\}^4 = \{a,b\}^4 \setminus \{(a,a,b,b)\}$. 
If $S$ can be defined by a quantifier-free primitive positive formula $\Phi$ over $C^\bullet$, then at least one of the equalities in $\Phi$ defines the set $\{a,b\}^4 \setminus \{(a,a,b,b)\}$ when restricted to the sublattice $\{ a,b \}$, and this leads to a contradiction using the same argument as in the proof of Lemma~\ref{lem:01 lattice not polhom}.
\end{proof}

\subsection{Abelian groups} \label{subsect:groups}

For Abelian groups all six conditions of Figure~\ref{fig:conditions} are equivalent, and these groups have already been determined, so we only need to combine some results from the literature to prove the following theorem.

\begin{theorem}
If $\alg$ is a finite Abelian group and $C=\Clo(\alg)$, then the following conditions are equivalent:
\begin{enumerate}[label=(\roman*)]
\item $\alg$ has property \sdc;
\item $\alg$ is homomorphism\hyp{}homogeneous;
\item $\alg$ is \polhom;
\item $\strA{\Ckarika}$ is \polhom;
\item $\strA{\Cpotty}$ is \polhom;
\item $\alg$ is injective in $\SPfin(\alg)$;
\item $\alg$ is injective in $\HSP(\alg)$;
\item each Sylow-subgroup of $\alg$ is homocyclic, i.e., $\alg \cong \Z_{q_1}^{m_1} \times \dots \times \Z_{q_k}^{m_k}$, where $q_1,\dots,q_k$ are powers of different primes and $m_1,\dots,m_k \in \N$.
\end{enumerate}
\end{theorem}

\begin{proof}
Conditions (i), (iii), (iv) and (vi) are equivalent by Corollary~\ref{cor:equivalences}.
By Proposition~\ref{prop:pol-hom Cpotty ==> pol-hom Ckarika}, (iv) is equivalent to (v), since we have $\genwrcl{\Cpotty} = \genwrcl{\Ckarika}$ for groups: every equality can be written in an equivalent form where there is only a single variable on the right hand side.
The equivalence of (ii) and (viii) follows from the description of quasi-injective Abelian groups presented as an exercise in \cite{Fuchs1970} (for finite groups quasi-injectivity is equivalent to homomorphism\hyp{}homogeneity).
The class of groups given in (viii) is closed under taking finite direct powers, so we can conclude with the help of Proposition~\ref{prop:direct powers} that (iii) and (viii) are equivalent.
It seems to be a folklore fact that the injective members of the variety of Abelian groups defined by the identity $nx=0$ with $n=q_1 \cdot\ldots\cdot q_k$ are exactly the groups given by (viii) (see, e.g., \cite{GarciaLarrion1982}).
Therefore, (viii) implies (vii), and this completes the proof, as (vii) trivially implies (vi).
\end{proof}

\subsection{Monounary algebras} \label{subsect:monounary}

A \emph{monounary algebra} is an algebra $\alg = \strA{f}$ with a single unary operation $f \in \OA^{(1)}$.
An element $a \in A$ is \emph{cyclic} if there is a natural number $k$ such that $f^k(a)=a$. 
(Here $f^k(a)$ stands for $f( \cdots f(a) \cdots )$ with a $k$-fold repetition of $f$, and we also use the convention $f^0(a)=a$.)
If $A$ is finite, then for every element $a \in A$ there is a least nonnegative integer $\h(a)$, called the \emph{height} of $a$, such that $f^{\h(a)}(a)$ is cyclic.
If $a \in A \setminus f(A)$, i.e., $a$ has no preimage, then we say that $a$ is a \emph{source}. 
(Note that $\h(a)=0$ if and only if $a$ is cyclic; in particular, $\h(a) \geq 1$ for any source $a$.)

Polymorphism\hyp{}homogeneous monounary algebras were characterized by Z.~Farkasov\'{a} and D.~Ja\-ku\-b\'{\i}\-ko\-v\'{a}\hyp{}Stu\-de\-novs\-k\'{a} in \cite{FarkasovaJakubikova2015} using Proposition~\ref{prop:direct powers} and the description of homomorphism\hyp{}homogeneous monounary algebras obtained by \'{E}.~Jung\'{a}bel and D.~Ma\v{s}ulovi\'{c} \cite{JungabelMasulovic2013}.
As an illustration of the results of Section~\ref{sect:equivalences}, we present a simple self-contained proof, which relies on the following technical lemma about quantifier elimination in monounary algebras.

\begin{lemma}\label{lem:monounary QE}
Let $\alg = \strA{f}$ be a finite monounary algebra, and let $C=\Clo(\alg)$.
The algebra $\alg$ is \polhom{} if and only if for each natural number $k$, there is a quantifier-free primitive positive formula $\Psi_k(x)$ over $\Ckarika$ such that
\begin{equation}\label{eq:Psi}
\forall a \in A \colon \Psi_k(a) \iff \exists a_0 \in A \colon a = f^k(a_0).
\end{equation}    
\end{lemma}

\begin{proof}
We use Theorem~\ref{thm:pol-hom Ckarika <==> pol-hom alg}: we prove that the existence of $\Psi_k$ is necessary and sufficient for polymorphism\hyp{}homogeneity of $\strA{\Ckarika}$. 
By Proposition~\ref{prop:QE <==> pol-hom}, the necessity is obvious; to prove sufficiency, let us consider an arbitrary primitive positive formula $\Phi(x_1,\dots,x_n)$ over $\Ckarika$.
We show how to eliminate one quantifier; repeatedly applying this procedure we can eliminate all quantifiers from $\Phi$.
So we may assume without loss of generality that $\Phi$ involves only one quantifier, hence it has the following form:
\begin{equation*}\label{eq:Phi}
\Phi(  x_{1},\dots,x_{n}) = \exists y \bigwith_{i=1}^{t} \bigl( f^{r_i}(u_i) = f^{s_i}(v_i) \bigr),
\end{equation*}
where $t,r_i,s_i$ are nonnegative integers, and the variables $u_i,v_i$ belong to the set $\{ x_1,\dots,x_n,y \}$ for $i=1,\dots,t$.
We define the weight of $\Phi$ as
\[
w(\Phi) = \sum_{\substack{i=1,\dots,t \\ u_i=y}}(r_i+1) + \sum_{\substack{i=1,\dots,t \\ v_i=y}}(s_i+1).
\]
Informally speaking, $w(\Phi)$ shows how ``deeply'' $y$ is involved in $\Phi$.

If $y$ occurs in at least two equalities in $\Phi$, then we can use (at least) one of the following four types of substitutions to decrease the weight of the formula (we omit trivial equalities):
\begin{equation*}
\arraycolsep=0.1pt
\renewcommand{\arraystretch}{1.5}
\begin{array}{rcl@{\hspace{1pt}}c@{\hspace{1pt}}rcl@{\hspace{2pt}}c@{\hspace{2pt}}rcl@{\hspace{1pt}}c@{\hspace{1pt}}rcl@{\hspace{3pt}}l}
f^{k}(y) & = & f^{\ell}(x_i) & \with & f^{m}(y) & = & f^{n}(x_j) & \rightsquigarrow & f^{k-m+n}(x_j) & = & f^\ell(x_i) & \with & f^{m}(y) & = & f^{n}(x_j), & \text{ if } k \geq m; \\
f^{k}(y) & = & f^{\ell}(x_i) & \with & f^{m}(y) & = & f^{n}(y) & \rightsquigarrow & f^{k-m+n}(y) & = & f^\ell(x_i) & \with & f^{m}(y) & = & f^{n}(y), &\text{ if } k \geq m, \ m > n;\\
f^{k}(y) & = & f^{\ell}(x_i) & \with & f^{m}(y) & = & f^{n}(y) & \rightsquigarrow & f^{m-k+\ell}(x_i) & = & f^{n}(y) & \with & f^{k}(y) & = & f^\ell(x_i), &\text{ if } k < m, \ m > n;\\
f^{k}(y) & = & f^{\ell}(y) & \with & f^{m}(y) & = & f^{n}(y) & \rightsquigarrow & f^{k-m+n}(y) & = & f^\ell(y) & \with & f^{m}(y) & = & f^{n}(y), &\text{ if } k \geq m, \ k > \ell, \ m > n.
\end{array}
\end{equation*}
After finitely many steps we arrive at a formula $\Phi'$ such that $\Phi'$ is equivalent to $\Phi$, and it is not possible to decrease the weight of $\Phi'$ any more using the substitutions above. 
This implies that the variable $y$ appears in at most one equality in $\Phi'$.
We have one of the following three cases for $\Phi'$.
\begin{enumerate}
\item If $y$ does not appear at all, then we can simply drop the quantifier $\exists y$ from $\Phi'$.
\item If $y$ appears in an equality of the form $f^{k}(y) = f^{\ell}(y)$, then there is no ``interaction'' between $y$ and the other variables. 
If there is an element $a \in A$ such that $f^{k}(a) = f^{\ell}(a)$, then we can again omit the quantifier $\exists y$ and the equality $f^{k}(y) = f^{\ell}(y)$ from $\Phi'$, and the resulting quantifier-free formula is equivalent to $\Phi'$ (hence also equivalent to $\Phi$).
If there is no element $a \in A$ such that $f^{k}(a) = f^{\ell}(a)$, then $\Phi'(x_1,\dots,x_n)$ is never satisfied: it defines the empty $n$-ary relation. 
In this case the empty relation can be defined by the quantifier-free formula $f^{k}(x_1) = f^{\ell}(x_1)$, thus this formula is equivalent to $\Phi$.
\item If $y$ appears in an equality of the form $f^{k}(y) = f^{\ell}(x_i)$, then let $\Phi''(x_1,\dots,x_n)$ be the formula that is obtained from $\Phi'$ by deleting the quantifier $\exists y$ and the equality $f^{k}(y) = f^{\ell}(x_i)$.
Then the quantifier-free formula $\Phi''(x_1,\dots,x_n) \with \Psi_k(f^{\ell}(x_i))$ is equivalent to $\Phi$, according to \eqref{eq:Psi}.\qedhere
\end{enumerate}
\end{proof}

\begin{theorem}[\cite{FarkasovaJakubikova2015}]\label{thm:monounary polhom}
If $\alg = \strA{f}$ is a finite monounary algebra and $C=\Clo(\alg)$, then the following conditions are equivalent:
\begin{enumerate}[label=(\roman*)]
\item $\alg$ has property \sdc;
\item $\alg$ is \polhom;
\item $\strA{\Ckarika}$ is \polhom;
\item $\alg$ is injective in $\SPfin(\alg)$;
\item Either $\alg$ has no sources, or all sources of $\alg$ have the same height: $\forall a,b \in A \setminus f(A)\colon \h(a)=\h(b)$.
\end{enumerate}
\end{theorem}

\begin{proof}
Conditions (i)--(iv) are equivalent by Corollary~\ref{cor:equivalences}, so it suffices to prove the equivalence of (ii) and (v).
As a preliminary observation, let us note that an element $a \in A$ is cyclic if and only if $a=f^{\ell}(a)$, where $\ell$ is the least common multiple of the lengths of the cycles of $\alg$.

Suppose that (v) holds, and assume first that there are no sources in $\alg$. Then every element is cyclic, thus the formula $x=x$ can be chosen for $\Psi_k(x)$ in \eqref{eq:Psi} for all $k \in \N$. Suppose now that all sources in $\alg$ have the same height $n \geq 1$. 
If $k \geq n$, then an arbitrary element $a \in A$ can be written as $a=f^k(a_0)$ for a suitable $a_0 \in A$ if and only if $a$ is cyclic. Thus, the formula $x = f^{\ell}(x)$ can be chosen for $\Psi_k(x)$, whenever $k \geq n$. Similarly, if $k<n$, then $\exists a_0 \in A\colon a=f^k(a_0)$ holds if and only if $\h(a) \leq n-k$, i.e., if $f^{n-k}(a)$ is cyclic. Therefore, we can put $f^{n-k}(x)=f^{n-k+\ell}(x)$ for $\Psi_k(x)$ in this case. This proves that (v) implies (ii), according to Lemma~\ref{lem:monounary QE}.

Conversely, assume that there exist formulas $\Psi_k(x)$ satisfying \eqref{eq:Psi}.
We can write $\Psi_1(x)$ in the following form, and we can assume without loss of generality that $r_i < s_i$ for $i=1,\dots,t$:
\begin{equation*}\label{eq:Psi1}
\Psi_1(x) = \bigwith_{i=1}^{t} \bigl( f^{r_i}(x) = f^{s_i}(x) \bigr).
\end{equation*}
If $a$ is a cyclic element, then $\Psi_1(a)$ must hold according to \eqref{eq:Psi}.
This implies that $r_i \equiv s_i \pmod{\ell}$ for all $i \in \{ 1.\dots,t \}$.

Now suppose for contradiction that there exist sources $a,b \in A \setminus f(A)$ with $\h(a) < \h(b)$.
Clearly, $\Psi_1(f(b))$ must be true by \eqref{eq:Psi}, thus we have $f^{r_i}(f(b)) = f^{s_i}(f(b))$ for each $i$.
This is equivalent to $f^{r_i+1}(b) = f^{s_i+1}(b)$, which implies that $f^{r_i+1}(b)$ is a cyclic element (recall that $r_i < s_i$), hence we have $\h(b) \leq r_i+1$.
Since $\h(a) \leq \h(b)-1$, we can conclude that $\h(a) \leq r_i$, i.e., $f^{r_i}(a)$ is a cyclic element.
The length of every cycle is a divisor of $\ell$, and we know that $r_i \equiv s_i \pmod{\ell}$, thus $f^{r_i}(a) = f^{s_i}(a)$.
This means that $\Psi_1(a)$ is true, contradicting the fact that $a$ is a source.
\end{proof}

Next we determine finite monounary algebras corresponding to the top left box of Figure~\ref{fig:conditions}.

\begin{theorem}\label{thm:monounary Cpotty polhom}
Let $\alg = \strA{f}$ be a finite monounary algebra, and let $C=\Clo(\alg)$. Then the relational structure $\strA{\Cpotty}$ is \polhom{} if and only if $f$ is either bijective or constant.
\end{theorem}

\begin{proof}
If $f$ is constant, then it is clear that $\strA{\Cpotty}$ is \polhom.
Assume now that $f$ is bijective. Then the condition of Theorem~\ref{thm:monounary polhom} is satisfied (there are no sources at all), so $\alg$ is \polhom, and thus by Theorem~\ref{thm:pol-hom Ckarika <==> pol-hom alg} $\strA{\Ckarika}$ is \polhom{} as well.
Therefore, by Proposition~\ref{prop:pol-hom Cpotty ==> pol-hom Ckarika}, it suffices to show that $\genwrcl{\Ckarika} = \genwrcl{\Cpotty}$.
This is clear, as any equality of the form $f^{k}(x)=f^{\ell}(y)$ with $k<\ell$ is equivalent to $x=f^{\ell-k}(y)$, since $f$ is bijective (here $x$ and $y$ might be the same variable).

For the other direction, let us suppose that $\strA{\Cpotty}$ is \polhom.
By Proposition~\ref{prop:pol-hom Cpotty ==> pol-hom Ckarika}, $\strA{\Ckarika}$ is also \polhom, and then Theorem~\ref{thm:monounary polhom} (together with Theorem~\ref{thm:pol-hom Ckarika <==> pol-hom alg}) implies that either there are no sources, or there is an integer $n \geq 1$ such that every source in $\alg$ has height $n$.
If there are no sources in $\alg$, then every element is cyclic, and therefore $f$ is bijective. From now on let us suppose that $\alg$ has sources with a common height $n$.
Proposition~\ref{prop:pol-hom Cpotty ==> pol-hom Ckarika} shows that there exists a quantifier-free primitive positive formula $\Phi(x,y)$ over $\Cpotty$ such that $\Phi(x,y)$ is equivalent to $f(x)=f(y)$.
We can write $\Phi(x,y)$ in the following form:
\begin{equation*}
\Phi(x,y) = \bigwith_{i=1}^{t} \bigl( f^{r_i}(u_i) = v_i \bigr),
\end{equation*}
where $t,r_i$ are nonnegative integers, and $u_i,v_i \in \{ x,y \}$ for $i=1,\dots,t$.
Obviously, $\Phi(a,a)$ holds for every element $a \in A$.
Let us choose $a$ to be of height $n$, i.e., let $a$ be a source.
Then $f^{r_i}(u_i) = v_i$ holds for $u_i=v_i=a$ if and only if $r_i=0$, thus $\Phi(x,y)$ is equivalent either to $x=y$ or to $x=x$.
Taking into account that $\Phi(x,y)$ is also equivalent to $f(x)=f(y)$, we can conclude that $f(x)=f(y) \iff x=y$ or $f(x)=f(y) \iff x=x$.
In the first case $f$ is a bijection, and in the second case $f$ is constant.
\end{proof}

Injective objects in the category of all monounary algebras were determined by D.~Ja\-ku\-b\'{\i}\-ko\-v\'{a}\hyp{}Stu\-de\-novs\-k\'{a} \cite{Jakubikova1998}; in the finite case these are exactly the monounary algebras $\alg = \strA{f}$ where $f$ is bijective and has a fixed point. 
However, in order to complete the picture of Figure~\ref{fig:conditions} for monounary algebras, we need to describe those monounary algebras $\alg$ that are injective in the variety $\HSP\alg$.
This has been done by D.~Ja\-ku\-b\'{\i}\-ko\-v\'{a}\hyp{}Stu\-de\-novs\-k\'{a} and G.~Cz\'{e}dli, but this result appeared only in Hungarian in the masters thesis \cite{Jeges2000} of T.~Jeges, a student of G.~Cz\'{e}dli.

\begin{theorem}[\cite{Jeges2000}]\label{thm:monounary injective}
A finite monounary algebra $\alg = \strA{f}$ is injective in the variety $\HSP(\alg)$ if and only if all of its sources have the same height and it has a one-element subalgebra (i.e., $f$ has a fixed point). 
\end{theorem}

Let us note that comparing theorems~\ref{thm:monounary polhom}, \ref{thm:monounary Cpotty polhom} and \ref{thm:monounary injective}, one can construct examples illustrating each one of the ``non-implications'' of Figure~\ref{fig:conditions}.  

\acknowledgements\label{sec:ack}

The authors are grateful to Dragan Ma\v{s}ulovi\'{c} and Christian Pech for helpful discussions.


\end{document}